\documentclass[12pt]{article}
\usepackage{paralist,amsmath,amsthm,amssymb,mathtools,url,graphicx,pdfpages,tikz,pdfpages,rotating}
\usepackage{cancel}
\usepackage[affil-it]{authblk}
\usepackage[left=1in,right=1in,top=1.2in, bottom=1in]{geometry}
\usepackage{mathdots}
\usepackage{comment}
\usepackage{bm} 

\newtheorem{theorem}{Theorem}[section]

\newtheorem{corollary}[theorem]{Corollary}

\theoremstyle{definition}

\newtheorem{example}[theorem]{Example}

\newtheorem{prop}[theorem]{Proposition}

\numberwithin{equation}{section}

\DeclareMathOperator{\circulant}{circ}

\newcommand{\blank}{$\;\;$}

\begin{document}
\title{Incidence and Laplacian matrices of wheel graphs and their inverses}
\author{Jerad Ipsen} 
\author{Sudipta Mallik} 
\affil{\small Department of Mathematics and Statistics, Northern Arizona University, 801 S. Osborne Dr.\\ PO Box: 5717, Flagstaff, AZ 86011, USA 

jli42@nau.edu, sudipta.mallik@nau.edu}

\maketitle
\begin{abstract}
It has been an open problem to find the Moore-Penrose inverses of the incidence, Laplacian, and  signless Laplacian matrices of families of graphs except trees and unicyclic graphs. Since the inverse formulas for an odd unicyclic graph and an even unicyclic graph are quite different, we consider wheel graphs as they are formed  from odd or even cycles. In this article we solve the open problem for wheel graphs. This work has an interesting connection to inverses of circulant matrices.

\end{abstract}

\section{Introduction}
Let $G$ be a simple graph on $n$ vertices $1,2,\ldots,n$ and $m$ edges $e_1,e_2,\ldots,e_m$ with the adjacency matrix $A$ and the degree matrix $D$. The {\it Laplacian matrix} $L$ and {\it signless Laplacian matrix} $Q$ of $G$ are defined as $L=D-A$ and $Q=D+A$ respectively. The vertex-edge {\it incidence matrix} $M$ of $G$ is the $n\times m$ matrix whose $(i,j)$-entry is $1$ if vertex $i$ is incident with edge $e_j$ and $0$ otherwise. It is well-known that $Q=MM^T$. An {\it oriented incidence matrix} $N$ of $G$ is the $n\times m$ matrix obtained from $M$ by changing one of the two 1s in each column of $M$ to $-1$. It is well-known that $Q=NN^T$ for any oriented incidence matrix $N$ of $G$.  \\

Circulant matrices play a crucial role in this article. A {\it circulant matrix} of order $n$ is an $n\times n$ matrix of the form
\[\left[\begin{array}{ccccc}
c_0 & c_1 & c_2 & \cdots & c_{n-1}\\
c_{n-1} & c_1 & c_2 & \cdots & c_{n-2}\\
c_{n-2} & c_{n-1} & c_1 & \cdots & c_{n-3}\\
\vdots & \vdots & \vdots & \ddots & \vdots\\
c_1 & c_2 & c_3 & \cdots & c_0
\end{array}\right]\]
which is denoted by $\circulant(c_0,c_1,\ldots,c_{n-1})$. For example, the incidence matrix of a cycle can be written as $\circulant(1,0,\ldots,0,1)$. The following are well-known properties of circulant matrices.

\begin{prop}\cite{{Searle1}}\label{circulant prop}\blank
\begin{enumerate}
    \item[(a)] Circulant matrices commute under multiplication.
    \item[(b)] The inverse of an invertible circulant matrix is a circulant matrix.
    \item[(c)] The inverse of  an invertible symmetric circulant matrix is a symmetric circulant matrix.
    \item[(d)] If $s$ is the row sum of  an invertible circulant matrix $C$, then $\frac{1}{s}$ is the row sum of $C^{-1}$.\\
\end{enumerate}
\end{prop}

The {\it Moore–Penrose inverse} of an $m\times n$ real matrix $A$, denoted by $A^+$, is the $n\times m$ real matrix that satisfies the following equations \cite{BG}:
$$AA^+A=A, A^+AA^+=A^+, (AA^+)^T=AA^+, (A^+A)^T=A^+A.$$
When $A$ is invertible, $A^+=A^{-1}$.\\
In 1965,  Ijira first studied the Moore-Penrose inverse of the oriented incidence matrix of a graph in \cite{I}. The same was done by Bapat for the Laplacian and edge-Laplacian of trees \cite{B}. Further research studied the same topic for different graphs such as distance regular graphs \cite{AB,ABE}. With the emergence of research on the signless Laplacian of graphs \cite{CRC, HM},  Hessert and Mallik studied the Moore-Penrose inverses of the incidence matrix and  signless Laplacian of a tree and an unicyclic graph in \cite{Hessert1, Hessert2}. It has been an open problem to find the Moore-Penrose inverses of the incidence, Laplacian, and  signless Laplacian matrices of other families of graphs. Note that the inverse formulas for an odd unicyclic graph and an even unicyclic graph are quite different \cite{Hessert2}. Since wheel graphs are formed  from odd or even cycles, they deserve to be investigated first for the inverse formulas of associated matrices. Recently an inverse formula for the distance matrix of a wheel graph has been studied by Balaji et al. \cite{Balaji}. In section 2, we study the Moore–Penrose inverses of the incidence and signless Laplacian matrices of the wheel graph on $n$ vertices. In section 3, we investigate the Moore–Penrose inverses of the oriented incidence and  Laplacian matrices of the wheel graph on $n$ vertices.

\section{Incidence and signless Laplacian matrices}
The wheel graph on $n\geq 4$ vertices, denoted by $W_n$, is obtained from an isolated vertex $v$ and a cycle on $n-1$ vertices by joining each vertex of the cycle to $v$. In this section first we study the Moore–Penrose inverse of the incidence matrix of $W_n$.

\begin{theorem}\label{M^+}
Let $W_n$ be the wheel graph on $n$ vertices with the incidence matrix $M$ given by
\[M=\left[\begin{array}{c|c}
\bm{1}^T & \bm{0}^T \\
\hline
I_{n-1} & C
\end{array}\right],\]
where $C$ is the circulant matrix $\circulant(1,0,...,0,1)$ of order $n-1$. The Moore-Penrose inverse of $M$ is given by
\[M^+=\frac{1}{2(n-1)} \left[\begin{array}{r|c}
2\bm{1} & X \\
\hline
-\bm{1} & Y
\end{array}\right],\]
where $X=2(CC^T+I_{n-1})^{-1}\left[ (n-1)I_{n-1}-J_{n-1}\right]$ and $Y= J_{n-1}+C^TX$.
\end{theorem}

\begin{proof}
First note that 
\[CC^T+I_{n-1}=\circulant(3,1,0,\ldots,0,1)\]
is strictly diagonally dominant and consequently invertible. Let
\[H=\frac{1}{2(n-1)} \left[\begin{array}{r|c}
2\bm{1} & X \\
\hline
-\bm{1} & Y
\end{array}\right],\]
where $X=2(CC^T+I_{n-1})^{-1}\left[ (n-1)I_{n-1}-J_{n-1}\right]$ and $Y= J_{n-1}+C^TX$.  We show that $H=M^+$.

\begin{align*}
MH &=\; \frac{1}{2(n-1)}\left[\begin{array}{c|c}
\bm{1}^T & \bm{0}^T \\
\hline
I_{n-1} & C
\end{array}\right]
 \left[\begin{array}{r|c}
2\bm{1} & X \\
\hline
-\bm{1} & Y
\end{array}\right]\\
&=\; \frac{1}{2(n-1)}\left[\begin{array}{c|c}
2\bm{1}^T\bm{1} & \bm{1}^{T}X \\
\hline
2I_{n-1}\bm{1}-C\bm{1} & I_{n-1}X+CY
\end{array}\right]\\
&=\; \frac{1}{2(n-1)}\left[\begin{array}{c|c}
2(n-1) & \bm{1}^{T}X \\
 \hline
2\bm{1}-2\bm{1} & X+CY
\end{array}\right]\\
&=\; \frac{1}{2(n-1)}\left[\begin{array}{c|c}
2(n-1) & \bm{1}^{T}X \\
 \hline
\bm{0} & X+CY
\end{array}\right] \tag{1}\label{eq1}
\end{align*}

Since the row sum of $CC^T+I_{n-1}=\circulant(3,1,0,\ldots,0,1)$ is $5$, $\bm{1}^{T}(CC^{T}+I_{n-1})^{-1}=\frac{1}{5}\bm{1}^T$ by Proposition \ref{circulant prop}. Then

\begin{eqnarray*}
\bm{1}^{T}X
&=& 2\bm{1}^{T}(CC^T+I_{n-1})^{-1}\left[ (n-1)I_{n-1}-J_{n-1}\right]\\
&=&  2\left(\frac{1}{5}\bm{1}^T\right)[(n-1)I_{n-1}-J_{n-1}]\\
&=&  \frac{2}{5}[(n-1)\bm{1}^{T}-(n-1)\bm{1}^{T}]\\
&=&  \bm{0}^T.
\end{eqnarray*}

Now we simplify $X+CY$ as follows.
\begin{eqnarray*}
X+CY &=& X+C(J_{n-1}+C^TX)\\
&=& X+CJ_{n-1}+CC^{T}X\\
&=& (I_{n-1}+CC^{T})X+CJ_{n-1}\\
&=& 2(CC^{T}+I_{n-1})(CC^{T}+I_{n-1})^{-1}[(n-1)I_{n-1}-J_{n-1}]+2J_{n-1}\\
&=& 2(n-1)I_{n-1}-2J_{n-1}+2J_{n-1}\\
&=& 2(n-1)I_{n-1}
\end{eqnarray*}

Putting $\bm{1}^{T}X = \bm{0}^T$ and $X+CY=2(n-1)I_{n-1}$ in (\ref{eq1}), we get
\[MH = \frac{1}{2(n-1)}\left[\begin{array}{c|c}
2(n-1) & \bm{0}^{T} \\
\hline
\bm{0} & 2(n-1)I_{n-1}
\end{array}\right]
=I_n.\]

Since $MH=I_n$,  we have $MHM=M$, $HMH=H$, and $(MH)^T=MH$. It remains to show that $HM$ is symmetric. 

\begin{align*}
    HM & =\; \frac{1}{2(n-1)} \left[\begin{array}{r|c}
2\bm{1} & X \\
\hline
-\bm{1} & Y
\end{array}\right] 
\left[\begin{array}{c|c}
\bm{1}^T & \bm{0} \\
\hline
I_{n-1} & C
\end{array}\right]\\
    &=\; \frac{1}{2(n-1)}\left[\begin{array}{c|c}
    2\bm{1}\bm{1}^T+X & XC \\
    \hline
    -\bm{1}\bm{1}^T+Y & YC
    \end{array}\right] \\
    & =\; \frac{1}{2(n-1)}\left[\begin{array}{c|c}
    2J_{n-1}+X & XC \\
    \hline
    -J_{n-1}+Y & YC
    \end{array}\right] \\
    & =\; \frac{1}{2(n-1)}\left[\begin{array}{c|c}
    2J_{n-1}+X & XC \\
    \hline
    C^TX & J_{n-1}C+C^TXC
    \end{array}\right]  \;(\text{since } Y= J_{n-1}+C^TX)\\
    & =\; \frac{1}{2(n-1)}\left[\begin{array}{c|c}
    2J_{n-1}+X & XC \\
    \hline
    C^TX & 2J_{n-1}+C^TXC
    \end{array}\right] \;(\text{since } J_{n-1}C= 2J_{n-1}) 
\end{align*}

To show $HM$ is symmetric, it suffices to show that $X$ is symmetric. Note that $CC^T+I_{n-1}$ is a symmetric circulant matrix and so is $(CC^T+I_{n-1})^{-1}$ by Proposition \ref{circulant prop}. Also $(n-1)I_{n-1}-J_{n-1}$ is a symmetric circulant matrix. Then so is 
\[X=2(CC^T+I_{n-1})^{-1}\left[ (n-1)I_{n-1}-J_{n-1}\right]\]
as a product of two symmetric circulant matrices.\\

Thus $H=M^+$.
\end{proof}

\begin{corollary}
In Theorem \ref{M^+}, $X$ is a symmetric circulant matrix and $Y$ is a circulant matrix.
\end{corollary}

\begin{example}
Consider $W_6$ with vertex and edge labeling given in Figure \ref{fig:W6} and its incidence matrix $M$. The Moore-Penrose inverse $M^+$ of $M$ is as follows.

\begin{figure}
\centering
\begin{tikzpicture}[shorten > = 1pt, auto, node distance = .5cm ]
\tikzset{vertex/.style = {shape = circle, draw, minimum size = 1em}}
\tikzset{edge/.style = {-}}
\node[vertex] (1) at (0,0){$1$};
\node[vertex] (2) at (0,3){$2$};
\node[vertex] (3) at (3,1){$3$};
\node[vertex] (4) at (2,-2.5){$4$};
\node[vertex] (5) at (-2,-2.5){$5$};
\node[vertex] (6) at (-3,1){$6$};
\draw[edge] (1) edge node[right]{$e_1$} (2);
\draw[edge] (1) edge node[below]{$e_2$} (3);
\draw[edge] (1) edge node[left]{$e_3$} (4);
\draw[edge] (1) edge node[left]{$e_4$} (5);
\draw[edge] (1) edge node[above]{$e_5$} (6);
\draw[edge] (2) edge node[above]{$e_6$} (3);
\draw[edge] (3) edge node[right]{$e_7$} (4);
\draw[edge] (4) edge node[below]{$e_8$} (5);
\draw[edge] (5) edge node[left]{$e_9$} (6);
\draw[edge] (6) edge node[above]{$e_{10}$} (2);
\end{tikzpicture}
\caption{$W_6$, the wheel graph on $6$ vertices}
\label{fig:W6}
\end{figure}
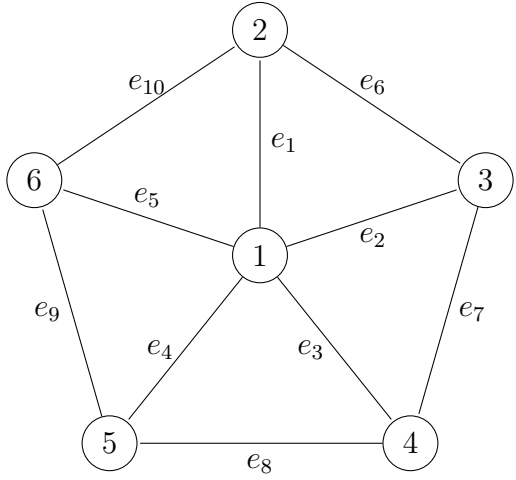

\[M=\left[\begin{array}{rrrrr|rrrrr}
1 & 1 & 1 & 1 & 1 & 0 & 0 & 0 & 0 & 0 \\
\hline
1 & 0 & 0 & 0 & 0 & 1 & 0 & 0 & 0 & 1 \\
0 & 1 & 0 & 0 & 0 & 1 & 1 & 0 & 0 & 0 \\
0 & 0 & 1 & 0 & 0 & 0 & 1 & 1 & 0 & 0 \\
0 & 0 & 0 & 1 & 0 & 0 & 0 & 1 & 1 & 0 \\
0 & 0 & 0 & 0 & 1 & 0 & 0 & 0 & 1 & 1
\end{array}\right],\;
M^{+}=\frac{1}{10}\left[\begin{array}{r|rrrrr}
2 & 4 & -2 & 0 & 0 & -2 \\
2 & -2 & 4 & -2 & 0 & 0 \\
2 & 0 & -2 & 4 & -2 & 0 \\
2 & 0 & 0 & -2 & 4 & -2 \\
2 & -2 & 0 & 0 & -2 & 4 \\
\hline
-1 & 3 & 3 & -1 & 1 & -1 \\
-1 & -1 & 3 & 3 & -1 & 1 \\
-1 & 1 & -1 & 3 & 3 & -1 \\
-1 & -1 & 1 & -1 & 3 & 3 \\
-1 & 3 & -1 & 1 & -1 & 3
\end{array}\right].\]
\end{example}

Theorem \ref{M^+} does not provide an explicit formula for each entry of $M^+$. To do that, we use the following result.

\begin{theorem}\cite[Theorem 1]{Searle1}\label{Searle1}
Let $n>3$ be an integer and $a,b,c$ real numbers such that $a^2>4bc$ and $b\neq0$. 
Except when $a+b+c=0$, or $n$ is even and $a=b+c$, 
\[[\circulant(a,b,0,0,...,0,c)]^{-1}=\circulant(a_0,a_1,\ldots,a_{n-1}), \]

where
\[a_j = \frac{z_1z_2}{b(z_1-z_2)}\left(\frac{z_1^j}{1-z_1^n}-\frac{z_2^j}{1-z_2^n}\right)\]

for $z_1,z_2=\left(-a\pm\sqrt{a^2-4bc}\right)/2c$.
\end{theorem}

\begin{corollary}
The inverse of the circulant matrix $\circulant(3,1,0,...,0,1)$ of order $n>3$ is given by
\[[\circulant(3,1,0,...,0,1)]^{-1}=\circulant(a_0,a_1,\ldots,a_{n-1}),\]
where 
\[a_j=\frac{2^{n-j}}{\sqrt{5}}\left[\frac{(-3+\sqrt{5})^j}{2^n-(-3+\sqrt{5})^n}-\frac{(-3-\sqrt{5})^j}{2^n-(-3-\sqrt{5})^n}\right].\]

\end{corollary}

\begin{proof}
Here $a=3$ and $b=c=1$. By Theorem \ref{Searle1},

\[a_j = \frac{z_1z_2}{b(z_1-z_2)}\left(\frac{z_1^j}{1-z_1^n}-\frac{z_2^j}{1-z_2^n}\right)\]

where $z_1,z_2=(-3\pm \sqrt{3^2-4\cdot 1 \cdot 1})/(2\cdot 1)=(-3\pm\sqrt{5})/2$. Then 

\begin{eqnarray*}
    a_j &=& \frac{\left(\frac{-3+\sqrt{5}}{2}\right) \left(\frac{-3-\sqrt{5}}{2}\right)}{1\left(\frac{-3+\sqrt{5}}{2}-\frac{-3-\sqrt{5}}{2}\right)}
    \left[\frac{\frac{(-3+\sqrt{5})^j}{2^j}}{1-\frac{(-3+\sqrt{5})^n}{2^n}}-\frac{\frac{(-3-\sqrt{5})^j}{2^j}}{1-\frac{(-3-\sqrt{5})^n}{2^n}}\right] \\
    &=& \frac{\frac{9-5}{4}}{\frac{2\sqrt{5}}{2}}\left[\frac{(-3+\sqrt{5})^j}{2^j\left(\frac{2^n-(-3+\sqrt{5})^n}{2^n}\right)}-\frac{(-3-\sqrt{5})^j}{2^j\left(\frac{2^n-(-3-\sqrt{5})^n}{2^n}\right)}\right] \\
    &=& \frac{1}{\sqrt{5}}\left[\frac{2^{n-j}(-3+\sqrt{5})^j}{2^n-(-3+\sqrt{5})^n}-\frac{2^{n-j}(-3-\sqrt{5})^j}{2^n-(-3-\sqrt{5})^n}\right] \\
    &=& \frac{2^{n-j}}{\sqrt{5}}\left[\frac{(-3+\sqrt{5})^j}{2^n-(-3+\sqrt{5})^n}-\frac{(-3-\sqrt{5})^j}{2^n-(-3-\sqrt{5})^n}\right].
\end{eqnarray*}
\end{proof}

\begin{corollary}\label{X=circ}
$X$ in Theorem \ref{M^+} is given by
$X=\circulant(b_0,b_1,\ldots,b_{n-2})$ where 
\[b_j=-\frac{2}{5}+\frac{2^{n-j}(n-1)}{\sqrt{5}}\left[\frac{(-3+\sqrt{5})^j}{2^{n-1}-(-3+\sqrt{5})^{n-1}}-\frac{(-3-\sqrt{5})^j}{2^{n-1}-(-3-\sqrt{5})^{n-1}}\right].\]
\end{corollary}
\begin{proof} 
Recall  $CC^T+I_{n-1}=\circulant(3,1,0,\ldots,0,1)$. Since the row sum of $CC^T+I_{n-1}$ is $5$, $(CC^{T}+I_{n-1})^{-1}J_{n-1}=\frac{1}{5}J_{n-1}$ by Proposition \ref{circulant prop}. Then

\begin{eqnarray*}
    X &=& 2(CC^T+I_{n-1})^{-1}[(n-1)I_{n-1}-J_{n-1}] \\
    &=& 2[\circulant(3,1,0,\ldots,0,1)]^{-1}[(n-1)I_{n-1}-J_{n-1}] \\
    &=& 2(n-1)[\circulant(3,1,0,\ldots,0,1)]^{-1}-2[\circulant(3,1,0,\ldots,0,1)]^{-1}J_{n-1} \\
    &=& 2(n-1)[\circulant(3,1,0,\ldots,0,1)]^{-1}-2\left(\frac{1}{5} J_{n-1}\right) \\
    &=& -\frac{2}{5}J_{n-1}+2(n-1)[\circulant(3,1,0,\ldots,0,1)]^{-1}.
\end{eqnarray*}

By the preceding corollary, $X=\circulant (b_0,b_1,...,b_{n-2})$ where 

\begin{eqnarray*}
    b_j &=& -\frac{2}{5}+2(n-1)\frac{2^{n-1-j}}{\sqrt{5}}\left[\frac{(-3+\sqrt{5})^j}{2^{n-1}-(-3+\sqrt{5})^{n-1}}-\frac{(-3-\sqrt{5})^j}{2^{n-1}-(-3-\sqrt{5})^{n-1}}\right] \\
    &=& -\frac{2}{5}+\frac{2^{n-j}(n-1)}{\sqrt{5}}\left[\frac{(-3+\sqrt{5})^j}{2^{n-1}-(-3+\sqrt{5})^{n-1}}-\frac{(-3-\sqrt{5})^j}{2^{n-1}-(-3-\sqrt{5})^{n-1}}\right].
\end{eqnarray*}
\end{proof}

\begin{corollary}
$Y$ in Theorem \ref{M^+} is given by
$Y=\circulant(d_0,d_1,\ldots,d_{n-2})$ where
\[d_0=\frac{1}{5}+\frac{4(n-1)}{\sqrt{5}}\left[\frac{2^{n-2}+(-3+\sqrt{5})^{n-2}}{2^{n-1}-(-3+\sqrt{5})^{n-1}}
    -\frac{2^{n-2}+(-3-\sqrt{5})^{n-2}}{2^{n-1}-(-3-\sqrt{5})^{n-1}}\right]\]
and for $j=1,2,\ldots,n-2$,
\[d_j=\frac{1}{5}+\frac{2^{n+1-j}(n-1)}{5+\sqrt{5}}\left[\frac{2(-3+\sqrt{5})^{j-1}}{2^{n-1}-(-3+\sqrt{5})^{n-1}}
-\frac{(-3-\sqrt{5})^{j}}{2^{n-1}-(-3-\sqrt{5})^{n-1}}\right].\]
\end{corollary}

\begin{proof}
Consider $X=\circulant(b_0,b_1,\ldots,b_{n-2})$ in Corollary \ref{X=circ}. Then 
\begin{eqnarray*}
    Y &=& J_{n-1}+C^TX \\
    &=& J_{n-1}+\circulant(b_{n-2}+b_0,b_0+b_1,...,b_{n-3}+b_{n-2})\\
    &=& \circulant(1+b_{n-2}+b_0,1+b_0+b_1,...,1+b_{n-3}+b_{n-2}).
\end{eqnarray*}

Then $Y=\circulant(d_0,d_1,\ldots,d_{n-2})$ where

\[d_j=1+b_j+b_{j-1},\;j=0,1,\ldots,n-2\; (\text{where $b_{-1}=b_{n-2}$}).
\]

\vspace*{-12pt}
\begin{eqnarray*}
d_0 &=& 1+b_{n-2}+b_0\\
&=& 1-\frac{2}{5}+\frac{2^{n-(n-2)}(n-1)}{\sqrt{5}}\left[\frac{(-3+\sqrt{5})^{n-2}}{2^{n-1}-(-3+\sqrt{5})^{n-1}}-\frac{(-3-\sqrt{5})^{n-2}}{2^{n-1}-(-3-\sqrt{5})^{n-1}}\right]\\
& & \;\;\;-\frac{2}{5}+\frac{2^{n}(n-1)}{\sqrt{5}}\left[\frac{1}{2^{n-1}-(-3+\sqrt{5})^{n-1}}-\frac{1}{2^{n-1}-(-3-\sqrt{5})^{n-1}}\right]\\
&=& \frac{1}{5} +\frac{4(n-1)}{\sqrt{5}}\left[\frac{(-3+\sqrt{5})^{n-2}+2^{n-2}}{2^{n-1}-(-3+\sqrt{5})^{n-1}}-\frac{(-3-\sqrt{5})^{n-2}+2^{n-2}}{2^{n-1}-(-3-\sqrt{5})^{n-1}}\right]\\
\end{eqnarray*}

For $j=1,2,\ldots,n-2$,
\begin{eqnarray*}
    d_j & = & 1+ b_j +b_{j-1} \\
    & = & 1+\frac{2^{n-j}(n-1)}{\sqrt{5}}\left[\frac{(-3+\sqrt{5})^j}{2^{n-1}-(-3+\sqrt{5})^{n-1}}-\frac{(-3-\sqrt{5})^j}{2^{n-1}-(-3-\sqrt{5})^{n-1}}\right]-\frac{2}{5} \\
    & & \;\;\;+\frac{2^{n-j+1}(n-1)}{\sqrt{5}}\left[\frac{(-3+\sqrt{5})^{j-1}}{2^{n-1}-(-3+\sqrt{5})^{n-1}}-\frac{(-3-\sqrt{5})^{j-1}}{2^{n-1}-(-3-\sqrt{5})^{n-1}}\right]-\frac{2}{5} \\
    & = & \frac{1}{5} + \frac{2^{n-j}(n-1)}{\sqrt{5}}\left[\frac{(-3+\sqrt{5})^j}{2^{n-1}-(-3+\sqrt{5})^{n-1}}-\frac{(-3-\sqrt{5})^j}{2^{n-1}-(-3-\sqrt{5})^{n-1}}\right] \\
    & & \;\;\;+\frac{2^{n-j}(n-1)}{\sqrt{5}}\left[\frac{2(-3+\sqrt{5})^{j-1}}{2^{n-1}-(-3+\sqrt{5})^{n-1}}-\frac{2(-3-\sqrt{5})^{j-1}}{2^{n-1}-(-3-\sqrt{5})^{n-1}}\right]\\
    & = & \frac{1}{5}+\frac{2^{n-j}(n-1)}{\sqrt{5}}\left[\frac{(-3+\sqrt{5}+2)(-3+\sqrt{5})^{j-1}}{2^{n-1}-(-3+\sqrt{5})^{n-1}}
    -\frac{(-3-\sqrt{5}+2)(-3-\sqrt{5})^{j-1}}{2^{n-1}-(-3-\sqrt{5})^{n-1}}\right]\\
    & = & \frac{1}{5}+\frac{2^{n-j}(n-1)}{\sqrt{5}}\left[\frac{(-1+\sqrt{5})(-3+\sqrt{5})^{j-1}}{2^{n-1}-(-3+\sqrt{5})^{n-1}}
    +\frac{(1+\sqrt{5})(-3-\sqrt{5})^{j-1}}{2^{n-1}-(-3-\sqrt{5})^{n-1}}\right]\\
    & = & \frac{1}{5}+\frac{2^{n-j}(n-1)}{\sqrt{5}(1+\sqrt{5})}\left[\frac{(1+\sqrt{5})(-1+\sqrt{5})(-3+\sqrt{5})^{j-1}}{2^{n-1}-(-3+\sqrt{5})^{n-1}}
    +\frac{(1+\sqrt{5})^2(-3-\sqrt{5})^{j-1}}{2^{n-1}-(-3-\sqrt{5})^{n-1}}\right]\\
    & = & \frac{1}{5}+\frac{2^{n-j}(n-1)}{\sqrt{5}+5}\left[\frac{4(-3+\sqrt{5})^{j-1}}{2^{n-1}-(-3+\sqrt{5})^{n-1}}
    +\frac{2(3+\sqrt{5})(-3-\sqrt{5})^{j-1}}{2^{n-1}-(-3-\sqrt{5})^{n-1}}\right]\\
    & = & \frac{1}{5}+\frac{2^{n+1-j}(n-1)}{5+\sqrt{5}}\left[\frac{2(-3+\sqrt{5})^{j-1}}{2^{n-1}-(-3+\sqrt{5})^{n-1}}
    -\frac{(-3-\sqrt{5})^{j}}{2^{n-1}-(-3-\sqrt{5})^{n-1}}\right].
\end{eqnarray*}
\end{proof}


Now we study the Moore–Penrose inverse of the signless Laplacian matrix of $W_n$.

\begin{theorem}
Let $W_n$ be the wheel graph on $n$ vertices with the signless Laplacian matrix $Q$ given by
\[Q=\left[\begin{array}{c|c}
n-1 & \bm{1}^T \\
\hline
\bm{1} & B
\end{array}\right],\]
where $B$ is the circulant matrix $\circulant(3,1,0,...,0,1)$ of order $n-1$. The Moore-Penrose inverse of $Q$ is given by
\[Q^+=\frac{1}{4(n-1)}
\left[\begin{array}{r|c}
    5 & -\bm{1}^T \\
    \hline -\bm{1} & J_{n-1}+2X
    \end{array}\right],\]
where $X=2(CC^T+I_{n-1})^{-1}\left[ (n-1)I_{n-1}-J_{n-1}\right]=\circulant(b_0,b_1,\ldots,b_{n-1})$ with
\[b_j=-\frac{2}{5}+\frac{2^{n-j}(n-1)}{\sqrt{5}}\left[\frac{(-3+\sqrt{5})^j}{2^{n-1}-(-3+\sqrt{5})^{n-1}}-\frac{(-3-\sqrt{5})^j}{2^{n-1}-(-3-\sqrt{5})^{n-1}}\right].\]
\end{theorem}

\begin{proof}
First note that $Q=MM^T$ for the incidence matrix $M$ of the form
\[M=\left[\begin{array}{c|c}
\bm{1}^T & \bm{0}^T \\
\hline
I_{n-1} & C
\end{array}\right],\]
where $C$ is the circulant matrix $\circulant(1,0,...,0,1)$ of order $n-1$. By Theorem \ref{M^+},
\[M^+=\frac{1}{2(n-1)} \left[\begin{array}{r|c}
2\bm{1} & X \\
\hline
-\bm{1} & Y
\end{array}\right],\]
where $X=2(CC^T+I_{n-1})^{-1}\left[ (n-1)I_{n-1}-J_{n-1}\right]$ and $Y= J_{n-1}+C^TX$.

\begin{align*}
Q^+ &=\; (MM^T)^+\\
&=\; (M^+)^TM^+\\
&=\; \frac{1}{4(n-1)^2}\left[\begin{array}{c|c}
    2\bm{1}^T & -\bm{1}^T \\
    \hline X^T & Y^T
    \end{array}\right]
    \left[\begin{array}{c|c}
    2\bm{1} & X \\
    \hline -\bm{1} & Y
    \end{array}\right] \\
&=\; \frac{1}{4(n-1)^2}\left[\begin{array}{c|c}
    2\bm{1}^T & -\bm{1}^T \\
    \hline X & Y^T
    \end{array}\right]
    \left[\begin{array}{c|c}
    2\bm{1} & X \\
    \hline -\bm{1} & Y
    \end{array}\right] \;\;(\text{since $X$ is symmetric}) \\
&=\; \frac{1}{4(n-1)^2}\left[\begin{array}{c|c}
    4\bm{1}^T\bm{1}+\bm{1}^T\bm{1} & 2\bm{1}^TX-\bm{1}^TY \\
    \hline
    2X\bm{1}-Y^T\bm{1} & X^2+Y^TY
    \end{array}\right] \\
&=\; \frac{1}{4(n-1)^2}\left[\begin{array}{c|c}
    4(n-1)+(n-1) & 2\bm{1}^TX-\bm{1}^T[J_{n-1}+C^TX] \\
    \hline
    2X\bm{1}-[J_{n-1}+XC]\bm{1} & X^2+Y^TY
    \end{array}\right] \\
&=\; \frac{1}{4(n-1)^2}\left[\begin{array}{c|c}
    5(n-1) & 2\bm{1}^TX-(n-1)\bm{1}^T-2\bm{1}^TX \\
    \hline
    2X\bm{1}-(n-1)\bm{1}-2X\bm{1} & X^2+Y^TY
    \end{array}\right] \\
&=\; \frac{1}{4(n-1)^2}\left[\begin{array}{c|c}
    5(n-1) & -(n-1)\bm{1}^T \\
    \hline
    -(n-1)\bm{1} & X^2+Y^TY\\
    \end{array}\right] \tag{2}\label{eq2}
\end{align*}

Now we simplify $X^2+Y^TY$ as follows.
\begin{align*}
&\;\;\;\;\; X^2+Y^TY\\
& =\; X^2+(J_{n-1}+C^TX)^T (J_{n-1}+C^TX) \\
& = X^2+(J_{n-1}+XC)(J_{n-1}+C^TX) \;\;(\text{since $X$ is symmetric}) \\
& =\; X^2+J_{n-1}^2+J_{n-1}C^TX+XCJ_{n-1}+XCC^TX \\
& = \left( XI_{n-1}X+XCC^TX\right)+J_{n-1}^2+J_{n-1}XC^T+XJ_{n-1}C \\
& =\; X(I_{n-1}+CC^T)X+(n-1)J_{n-1} \;\;(\text{since $J_{n-1}X=XJ_{n-1}=0$}) \\
& =\; 2X[(n-1)I_{n-1}-J_{n-1}]+(n-1)J_{n-1} 
\;\;(\text{since $(CC^T+I_{n-1})X=2\left[ (n-1)I_{n-1}-J_{n-1}\right]$}) \\
& =\; (n-1)2X-2XJ_{n-1}+(n-1)J_{n-1} \;\;(\text{since $XJ_{n-1}=0$}) \\
& =\; (n-1)[J_{n-1}+2X]
\end{align*}

Plugging $X^2+Y^TY=(n-1)[J_{n-1}+2X]$ in (\ref{eq2}), we get
\begin{eqnarray*}
Q^+ &=& \frac{1}{4(n-1)^2}\left[\begin{array}{c|c}
    5(n-1) & -(n-1)\bm{1}^T \\
    \hline
    -(n-1)\bm{1} & (n-1)[J_{n-1}+2X]\\
    \end{array}\right] \\    
&=&  \frac{1}{4(n-1)}\left[\begin{array}{r|c}
    5 & -\bm{1}^T \\
    \hline -\bm{1} & J_{n-1}+2X
    \end{array}\right], 
\end{eqnarray*}
where $X$ is given by Corollary \ref{X=circ}.
\end{proof}

\begin{example}
Consider $W_6$ with vertex and edge labeling given in Figure \ref{fig:W6} and its signless Laplacian matrix $Q$. The Moore-Penrose inverse $Q^+$ of $Q$ is as follows.

\[Q=\left[\begin{array}{rrrrrr}
5 & 1 & 1 & 1 & 1 & 1 \\
1 & 3 & 1 & 0 & 0 & 1 \\
1 & 1 & 3 & 1 & 0 & 0 \\
1 & 0 & 1 & 3 & 1 & 0 \\
1 & 0 & 0 & 1 & 3 & 1 \\
1 & 1 & 0 & 0 & 1 & 3
\end{array}\right],\;
Q^+=\frac{1}{20}\left[\begin{array}{r|rrrrr}
5 & -1 & -1 & -1 & -1 & -1 \\
\hline
-1 & 9 & -3 & 1 & 1 & -3 \\
-1 & -3 & 9 & -3 & 1 & 1 \\
-1 & 1 & -3 & 9 & -3 & 1 \\
-1 & 1 & 1 & -3 & 9 & -3 \\
-1 & -3 & 1 & 1 & -3 & 9
\end{array}\right].\]
\end{example}

\section{Oriented incidence and Laplacian matrices}
\begin{theorem}\label{N^+}
Let $W_n$ be the wheel graph on $n$ vertices with the oriented incidence matrix $N$ given by
\[N=\left[\begin{array}{c|c}
\bm{1}^T & \bm{0}^T \\
\hline
-I_{n-1} & C
\end{array}\right],\]
where $C$ is the circulant matrix $\circulant(1,0,...,0,-1)$ of order $n-1$. The Moore-Penrose inverse of $N$ is given by
\[N^+=\frac{1}{n} \left[\begin{array}{r|c}
\bm{1} & X \\
\hline
\bm{0} & Y
\end{array}\right],\]
where $X=(CC^T+I_{n-1})^{-1}(J_{n-1}-nI_{n-1})$ and $Y=-C^TX$.
\end{theorem}

\begin{proof}
First note that 
\[CC^T+I_{n-1}=\circulant(3,-1,0,\ldots,0,-1)\]
is strictly diagonally dominant and consequently invertible. Let
\[H=\frac{1}{n} \left[\begin{array}{r|c}
\bm{1} & X \\
\hline
\bm{0} & Y
\end{array}\right],\]
where $X=(CC^T+I_{n-1})^{-1}(J_{n-1}-nI_{n-1})$ and $Y=-C^TX$. We show that $H=N^+$.

\begin{align*}
    NH &=\; \frac{1}{n}\left[\begin{array}{c|c}
    \bm{1}^T & \bm{0}^T \\
    \hline
    -I_{n-1} & C
    \end{array}\right]\left[\begin{array}{c|c}
    \bm{1} & X \\
    \hline
    \bm{0} & Y
    \end{array}\right] \\
    &=\; \frac{1}{n}\left[\begin{array}{c|c}
    \bm{1}^T\bm{1} & \bm{1}^TX \\
    \hline
    -\bm{1} & -X+CY
    \end{array}\right] \\
    &=\; \frac{1}{n}\left[\begin{array}{c|c}
    n-1 & \bm{1}^TX \\
    \hline
    -\bm{1} & -X+CY
    \end{array}\right] \tag{3}\label{eq3}
\end{align*}

Since the row sum of $CC^T+I_{n-1}=\circulant(3,-1,0,\ldots,0,-1)$ is $1$, $\bm{1}^{T}(CC^{T}+I_{n-1})^{-1}=\frac{1}{1}\bm{1}^T=\bm{1}^T$ by Proposition \ref{circulant prop}. Then

\begin{eqnarray*}
    \bm{1}^TX &=& \bm{1}^T(CC^T+I_{n-1})^{-1}(J_{n-1}-nI_{n-1}) \\
    &=& \bm{1}^T(J_{n-1}-nI_{n-1}) \\
    &=& (n-1)\bm{1}^T-n\bm{1}^T\\
    &=& -\bm{1}^T.
\end{eqnarray*}

Now we simplify $CY-X$ as follows.
\begin{eqnarray*}
    CY-X &=& C(-C^TX)-X \\
    &=& -CC^TX-X \\
    &=& -(CC^T+I_{n-1})X \\
    &=& -(CC^T+I_{n-1})(CC^T+I_{n-1})^{-1}(J_{n-1}-nI_{n-1}) \\
    &=& -(J_{n-1}-nI_{n-1}) \\
    &=& nI_{n-1}-J_{n-1}.
\end{eqnarray*}

Putting $\bm{1}^{T}X = -\bm{1}^{T}$ and $CY-X=nI_{n-1}-J_{n-1}$ in (\ref{eq3}), we get

\begin{eqnarray*}
NH &=& \frac{1}{n}\left[\begin{array}{c|c}
    n-1 & -\bm{1}^T \\
    \hline
    -\bm{1} & nI_{n-1}-J_{n-1}
    \end{array}\right]
=\left[\begin{array}{c|c}
1-\frac{1}{n} & -\frac{1}{n}\bm{1}^T \\
\hline
-\frac{1}{n}\bm{1} & I_{n-1}-\frac{1}{n}J_{n-1}
\end{array}\right] \;=\; I_n-\frac{1}{n}J_n.
\end{eqnarray*}

Now we show NHN=N.

\begin{eqnarray*}
    NHN &=& \left(I_n-\frac{1}{n}J_n\right)N \\
    &=& N-\frac{1}{n}J_nN \\
    &=& N \;\; (\text{since the column sum of N is 0})
\end{eqnarray*}

We also show HNH=H.

\begin{eqnarray*}
    HNH &=& H\left(I_n-\frac{1}{n}J_n\right) \\
    &=& H-\frac{1}{n}HJ_n
\end{eqnarray*}

To show $H-\frac{1}{n}HJ_n=H$, we show that $HJ_n=O$. Note that 

\begin{eqnarray*}
    HJ_n &=& \left[\begin{array}{c|c}
    \bm{1} & X \\
    \hline
    \bm{0} & Y
    \end{array}\right]\left[\begin{array}{c|c}
    \bm{1}^T & \bm{1}^T \\
    \hline
    J_{n-1} & J_{n-1}
    \end{array}\right] 
    =\left[\begin{array}{c|c}
    J_{n-1}+XJ_{n-1} & J_{n-1}+XJ_{n-1} \\
    \hline
    YJ_{n-1} & YJ_{n-1}
    \end{array}\right].
\end{eqnarray*}

To show $H-\frac{1}{n}HJ_n=H$, it suffices to show $J_{n-1}+XJ_{n-1}=O$ and $YJ_{n-1}=O$.

\begin{eqnarray*}
    J_{n-1}+XJ_{n-1} &=& J_{n-1}+(CC^T+I_{n-1})^{-1}(J_{n-1}-nI_{n-1})J_{n-1} \\
    &=& J_{n-1}+(CC^T+I_{n-1})^{-1}((n-1)J_{n-1}-nJ_{n-1}) \\
    &=& J_{n-1}+(CC^T+I_{n-1})^{-1}(-J_{n-1}) \\
    &=& J_{n-1}-J_{n-1} \;\; (\text{since the row sum of $(CC^T+I_{n-1})^{-1}$ is $1$}) \\
    &=& O
\end{eqnarray*}

\vspace*{-14pt}
\begin{eqnarray*}
    YJ_{n-1} &=& -C^TXJ_{n-1} \\
    &=& -C^T(CC^T+I_{n-1})^{-1}(J_{n-1}-nI_{n-1})J_{n-1} \\
    &=& -C^T(CC^T+I_{n-1})^{-1}(-J_{n-1}) \\
    &=& C^TJ_{n-1} \\
    &=& O\;\; (\text{since the row sum of $C^T$ is $0$})
\end{eqnarray*}

Note that $NH=I_n-\frac{1}{n}J_n$ is symmetric. It remains to show that $HN$ is symmetric.

\begin{eqnarray*}
    HN &=& \frac{1}{n}\left[\begin{array}{c|c}
    \bm{1} & X \\
    \hline
    \bm{0} & Y
    \end{array}\right]\left[\begin{array}{c|c}
    \bm{1}^T & \bm{0}^T \\
    \hline
    -I_{n-1} & C
    \end{array}\right] \\
    &=& \frac{1}{n}\left[\begin{array}{c|c}
    \bm{1}\bm{1}^T-X & XC \\
    \hline
    -Y & YC
    \end{array}\right] \\
    &=& \frac{1}{n}\left[\begin{array}{c|c}
    J_{n-1}-X & XC \\
    \hline
    -Y & YC
    \end{array}\right] \\
    &=& \frac{1}{n}\left[\begin{array}{c|c}
    J_{n-1}-X & XC \\
    \hline
    C^TX & -C^TXC
    \end{array}\right]
\end{eqnarray*}

To show $HN$ is symmetric, it suffices to show that $X$ is symmetric . Note that $CC^T+I_{n-1}$ is a symmetric circulant matrix and so is $(CC^T+I_{n-1})^{-1}$ by Proposition \ref{circulant prop}. Also $J_{n-1}-nI_{n-1}$ is a symmetric circulant matrix. Then so is 
\[X=(CC^T+I_{n-1})^{-1}\left[J_{n-1}-nI_{n-1}\right]\]
as a product of two symmetric circulant matrices.\\

Thus $H=N^+$.
\end{proof}

\begin{example}
Consider $W_6$ with vertex and edge labeling and edge orientation given in Figure \ref{fig:OW6} and its oriented incidence matrix $N$. The Moore-Penrose inverse $N^+$ of $N$ is as follows.

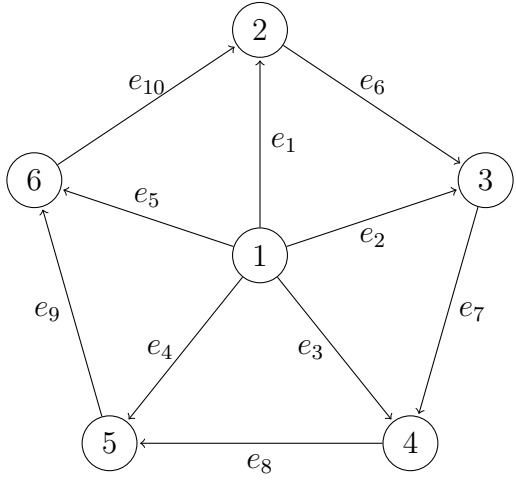
\begin{figure}
\centering
\begin{tikzpicture}[shorten > = 1pt, auto, node distance = .5cm ]
\tikzset{vertex/.style = {shape = circle, draw, minimum size = 1em}}
\tikzset{edge/.style = {-}}
\node[vertex] (1) at (0,0){$1$};
\node[vertex] (2) at (0,3){$2$};
\node[vertex] (3) at (3,1){$3$};
\node[vertex] (4) at (2,-2.5){$4$};
\node[vertex] (5) at (-2,-2.5){$5$};
\node[vertex] (6) at (-3,1){$6$};
\draw[->] (1) edge node[right]{$e_1$} (2);
\draw[->] (1) edge node[below]{$e_2$} (3);
\draw[->] (1) edge node[left]{$e_3$} (4);
\draw[->] (1) edge node[left]{$e_4$} (5);
\draw[->] (1) edge node[above]{$e_5$} (6);
\draw[->] (2) edge node[above]{$e_6$} (3);
\draw[->] (3) edge node[right]{$e_7$} (4);
\draw[->] (4) edge node[below]{$e_8$} (5);
\draw[->] (5) edge node[left]{$e_9$} (6);
\draw[->] (6) edge node[above]{$e_{10}$} (2);
\end{tikzpicture}
\caption{An oriented wheel graph on $6$ vertices}
\label{fig:OW6}
\end{figure}

\[N=\left[\begin{array}{rrrrr|rrrrr}
1 & 1 & 1 & 1 & 1 & 0 & 0 & 0 & 0 & 0 \\
\hline
-1 & 0 & 0 & 0 & 0 & 1 & 0 & 0 & 0 & -1 \\
0 & -1 & 0 & 0 & 0 & -1 & 1 & 0 & 0 & 0 \\
0 & 0 & -1 & 0 & 0 & 0 & -1 & 1 & 0 & 0 \\
0 & 0 & 0 & -1 & 0 & 0 & 0 & -1 & 1 & 0 \\
0 & 0 & 0 & 0 & -1 & 0 & 0 & 0 & -1 & 1
\end{array}\right],\]
\[N^+=\frac{1}{66}\left[\begin{array}{r|rrrrr}
11 & -19 & -1 & 5 & 5 & -1 \\
11 & -1 & -19 & -1 & 5 & 5 \\
11 & 5 & -1 & -19 & -1 & 5 \\
11 & 5 & 5 & -1 & -19 & -1 \\
11 & -1 & 5 & 5 & -1 & -19 \\
\hline
0 & 18 & -18 & -6 & 0 & 6 \\
0 & 6 & 18 & -18 & -6 & 0 \\
0 & 0 & 6 & 18 & -18 & -6 \\
0 & -6 & 0 & 6 & 18 & -18 \\
0 & -18 & -6 & 0 & 6 & 18
\end{array}\right].\]
\end{example}

Theorem \ref{N^+} does not provide an explicit formula for each entry of $N^+$. To do that, we use the following result.

\begin{corollary}
The inverse of the circulant matrix $\circulant(3,-1,0,...,0,-1)$ of order $n>3$ is given by
\[[\circulant(3,-1,0,...,0,-1)]^{-1}=\circulant(a_0,a_1,\ldots,a_{n-1}),\]
where 
\[a_j=\frac{2^{n-j}}{\sqrt{5}}\left[\frac{(3-\sqrt{5})^j}{2^n-(3-\sqrt{5})^n}-\frac{(3+\sqrt{5})^j}{2^n-(3+\sqrt{5})^n}\right].\]

\end{corollary}

\begin{proof}
Here $a=3$ and $b=c=-1$. By Theorem \ref{Searle1},

\[a_j = \frac{z_1z_2}{b(z_1-z_2)}\left(\frac{z_1^j}{1-z_1^n}-\frac{z_2^j}{1-z_2^n}\right)\]

where $z_1,z_2=(-3\pm \sqrt{3^2-4(-1)(-1)})/(2(-1))=(3\mp\sqrt{5})/2$. Then 

\begin{eqnarray*}
    a_j &=& \frac{\left(\frac{3-\sqrt{5}}{2}\right) \left(\frac{3+\sqrt{5}}{2}\right)}{-1\left(\frac{3-\sqrt{5}}{2}-\frac{3+\sqrt{5}}{2}\right)}
    \left[\frac{\frac{(3-\sqrt{5})^j}{2^j}}{1-\frac{(3-\sqrt{5})^n}{2^n}}-\frac{\frac{(3+\sqrt{5})^j}{2^j}}{1-\frac{(3+\sqrt{5})^n}{2^n}}\right] \\
    &=& \frac{\frac{9-5}{4}}{\sqrt{5}}\left[\frac{(3-\sqrt{5})^j}{2^j\left(\frac{2^n-(3-\sqrt{5})^n}{2^n}\right)}-\frac{(3+\sqrt{5})^j}{2^j\left(\frac{2^n-(3+\sqrt{5})^n}{2^n}\right)}\right] \\
    &=& \frac{1}{\sqrt{5}}\left[\frac{2^{n-j}(3-\sqrt{5})^j}{2^n-(3-\sqrt{5})^n}-\frac{2^{n-j}(3+\sqrt{5})^j}{2^n-(3+\sqrt{5})^n}\right] \\
    &=& \frac{2^{n-j}}{\sqrt{5}}\left[\frac{(3-\sqrt{5})^j}{2^n-(3-\sqrt{5})^n}-\frac{(3+\sqrt{5})^j}{2^n-(3+\sqrt{5})^n}\right].
\end{eqnarray*}
\end{proof}

\begin{corollary}\label{X=circ2}
$X$ in Theorem \ref{N^+} is given by
$X=\circulant(b_0,b_1,\ldots,b_{n-2})$ where 
\[b_j=1+\frac{n2^{n-1-j}}{\sqrt{5}}\left[\frac{(3+\sqrt{5})^j}{2^{n-1}-(3+\sqrt{5})^{n-1}}-\frac{(3-\sqrt{5})^j}{2^{n-1}-(3-\sqrt{5})^{n-1}}\right].\]
\end{corollary}
\begin{proof} 
Recall  $CC^T+I_{n-1}=\circulant(3,-1,0,\ldots,0,-1)$. Since the row sum of $CC^T+I_{n-1}$ is $1$, $(CC^{T}+I_{n-1})^{-1}J_{n-1}=J_{n-1}$ by Proposition \ref{circulant prop}. Then

\begin{eqnarray*}
    X &=& (CC^T+I_{n-1})^{-1}(J_{n-1}-nI_{n-1}) \\
    &=& [\circulant(3,-1,0,\ldots,0,-1)]^{-1}(J_{n-1}-nI_{n-1}) \\
    &=& [\circulant(3,-1,0,\ldots,0,-1)]^{-1}J_{n-1}-n[\circulant(3,-1,0,\ldots,0,-1)]^{-1} \\
    &=& J_{n-1}-n\circulant(3,-1,0,\ldots,0,-1)]^{-1}.
\end{eqnarray*}

By the preceding corollary, $X=\circulant (b_0,b_1,...,b_{n-2})$ where 

\begin{eqnarray*}
    b_j &=& 1-\frac{n2^{n-1-j}}{\sqrt{5}}\left[\frac{(3-\sqrt{5})^j}{2^{n-1}-(3-\sqrt{5})^{n-1}}-\frac{(3+\sqrt{5})^j}{2^{n-1}-(3+\sqrt{5})^{n-1}}\right] \\
    &=& 1+\frac{n2^{n-1-j}}{\sqrt{5}}\left[\frac{(3+\sqrt{5})^j}{2^{n-1}-(3+\sqrt{5})^{n-1}}-\frac{(3-\sqrt{5})^j}{2^{n-1}-(3-\sqrt{5})^{n-1}}\right].
\end{eqnarray*}
\end{proof}

\begin{corollary}
$Y$ in Theorem \ref{N^+} is given by
$Y=\circulant(d_0,d_1,\ldots,d_{n-2})$ where 
\[d_0=\frac{2n}{\sqrt{5}}\left[\frac{(3+\sqrt{5})^{n-2}-2^{n-2}}{2^{n-1}-(3+\sqrt{5})^{n-1}}   -\frac{(3-\sqrt{5})^{n-2}-2^{n-2}}{2^{n-1}-(3-\sqrt{5})^{n-1}}\right].\]
and for $j=1,2,\ldots,n-2$,
\[d_j=-\frac{n2^{n-j}}{5+\sqrt{5}}\left[\frac{(3+\sqrt{5})^{j}}{2^{n-1}-(3+\sqrt{5})^{n-1}}
+\frac{2(3-\sqrt{5})^{j-1}}{2^{n-1}-(3-\sqrt{5})^{n-1}}\right].\]
\end{corollary}

\begin{proof}
Consider $X=\circulant(b_0,b_1,\ldots,b_{n-2})$ in Corollary \ref{X=circ2}. Then 
\begin{eqnarray*}
    Y &=& -C^TX \\
    &=& -\circulant(b_0-b_{n-2},b_1-b_0,...,b_{n-2}-b_{n-3})\\
    &=& \circulant(b_{n-2}-b_0,b_0-b_1,...,b_{n-3}-b_{n-2}).
\end{eqnarray*}

Then $Y=\circulant(d_0,d_1,\ldots,d_{n-2})$ where

\[d_j=b_{j-1}-b_j,\;j=0,1,\ldots,n-2\; (\text{where $b_{-1}=b_{n-2}$}).
\]

\vspace*{-12pt}
\begin{eqnarray*}
    d_0 & = & b_{n-2}-b_0 \\
    &=& 1+\frac{n2^{n-1-(n-2)}}{\sqrt{5}}\left[\frac{(3+\sqrt{5})^{n-2}}{2^{n-1}-(3+\sqrt{5})^{n-1}}-\frac{(3-\sqrt{5})^{n-2}}{2^{n-1}-(3-\sqrt{5})^{n-1}}\right]\\
    && -1-\frac{n2^{n-1}}{\sqrt{5}}\left[\frac{1}{2^{n-1}-(3+\sqrt{5})^{n-1}}-\frac{1}{2^{n-1}-(3-\sqrt{5})^{n-1}}\right] \\
    &=& \frac{2n}{\sqrt{5}}\left[\frac{(3+\sqrt{5})^{n-2}-2^{n-2}}{2^{n-1}-(3+\sqrt{5})^{n-1}}   -\frac{(3-\sqrt{5})^{n-2}-2^{n-2}}{2^{n-1}-(3-\sqrt{5})^{n-1}}\right]
\end{eqnarray*}

For $j=1,2,\ldots,n-2$,
\begin{eqnarray*}
    d_j & = & b_{j-1}-b_j \\
    &=& 1+\frac{n2^{n-1-(j-1)}}{\sqrt{5}}\left[\frac{(3+\sqrt{5})^{j-1}}{2^{n-1}-(3+\sqrt{5})^{n-1}}-\frac{(3-\sqrt{5})^{j-1}}{2^{n-1}-(3-\sqrt{5})^{n-1}}\right] \\
    & & -1-\frac{n2^{n-1-j}}{\sqrt{5}}\left[\frac{(3+\sqrt{5})^j}{2^{n-1}-(3+\sqrt{5})^{n-1}}-\frac{(3-\sqrt{5})^j}{2^{n-1}-(3-\sqrt{5})^{n-1}}\right] \\
    &=& \frac{n2^{n-1-j}}{\sqrt{5}}\left[\frac{2(3+\sqrt{5})^{j-1}}{2^{n-1}-(3+\sqrt{5})^{n-1}}-\frac{2(3-\sqrt{5})^{j-1}}{2^{n-1}-(3-\sqrt{5})^{n-1}}\right] \\
    & & -\frac{n2^{n-1-j}}{\sqrt{5}}\left[\frac{(3+\sqrt{5})^j}{2^{n-1}-(3+\sqrt{5})^{n-1}}-\frac{(3-\sqrt{5})^j}{2^{n-1}-(3-\sqrt{5})^{n-1}}\right] \\
    &=& \frac{n2^{n-1-j}}{\sqrt{5}}\left[\frac{(2-(3+\sqrt{5}))(3+\sqrt{5})^{j-1}}{2^{n-1}-(3+\sqrt{5})^{n-1}}-\frac{(2-(3-\sqrt{5}))(3-\sqrt{5})^{j-1}}{2^{n-1}-(3-\sqrt{5})^{n-1}}\right] \\
    &=& \frac{n2^{n-1-j}}{\sqrt{5}}\left[\frac{-(1+\sqrt{5})(3+\sqrt{5})^{j-1}}{2^{n-1}-(3+\sqrt{5})^{n-1}}-\frac{(-1+\sqrt{5})(3-\sqrt{5})^{j-1}}{2^{n-1}-(3-\sqrt{5})^{n-1}}\right] \\
    &=& \frac{n2^{n-1-j}}{\sqrt{5}(1+\sqrt{5})}\left[-\frac{(1+\sqrt{5})^2(3+\sqrt{5})^{j-1}}{2^{n-1}-(3+\sqrt{5})^{n-1}}
    +\frac{(1+\sqrt{5})(1-\sqrt{5})(3-\sqrt{5})^{j-1}}{2^{n-1}-(3-\sqrt{5})^{n-1}}\right] \\
    &=& \frac{n2^{n-1-j}}{5+\sqrt{5}}\left[-\frac{(6+2\sqrt{5})(3+\sqrt{5})^{j-1}}{2^{n-1}-(3+\sqrt{5})^{n-1}}
    -\frac{4(3-\sqrt{5})^{j-1}}{2^{n-1}-(3-\sqrt{5})^{n-1}}\right]\\
    &=& -\frac{n2^{n-1-j}}{5+\sqrt{5}}\left[\frac{2(3+\sqrt{5})^{j}}{2^{n-1}-(3+\sqrt{5})^{n-1}}
    +\frac{4(3-\sqrt{5})^{j-1}}{2^{n-1}-(3-\sqrt{5})^{n-1}}\right]\\
    &=& -\frac{n2^{n-j}}{5+\sqrt{5}}\left[\frac{(3+\sqrt{5})^{j}}{2^{n-1}-(3+\sqrt{5})^{n-1}}
    +\frac{2(3-\sqrt{5})^{j-1}}{2^{n-1}-(3-\sqrt{5})^{n-1}}\right].
\end{eqnarray*}
\end{proof}

Now we study the Moore–Penrose inverse of the Laplacian matrix of $W_n$.

\begin{theorem}
Let $W_n$ be the wheel graph on $n$ vertices with the Laplacian matrix $L$ given by
\[L=\left[\begin{array}{c|c}
    n-1 & -\bm{1}^T \\
    \hline
    -\bm{1} & B
    \end{array}\right],\]
where $B$ is the circulant matrix $\circulant(3,-1,0,...,0,-1)$ of order $n-1$. The Moore-Penrose inverse of $L$ is given by
\[L^+=\frac{1}{n^2}\left[\begin{array}{c|c}
    n-1 & -\bm{1}^T \\
    \hline
    -\bm{1} & -J_{n-1}-nX
    \end{array}\right],\]
where $X=(CC^T+I_{n-1})^{-1}\left[J_{n-1}-nI_{n-1}\right]=\circulant(b_0,b_1,\ldots,b_{n-1})$ with 

\[b_j=1+\frac{n2^{n-1-j}}{\sqrt{5}}\left[\frac{(3+\sqrt{5})^j}{2^{n-1}-(3+\sqrt{5})^{n-1}}-\frac{(3-\sqrt{5})^j}{2^{n-1}-(3-\sqrt{5})^{n-1}}\right].\]

\end{theorem}

\begin{proof}
First note that $L=NN^T$ for the incidence matrix $N$ of the form
\[N=\left[\begin{array}{c|c}
\bm{1}^T & \bm{0}^T \\
\hline
-I_{n-1} & C
\end{array}\right],\]
where $C$ is the circulant matrix $\circulant(1,0,...,0,-1)$ of order $n-1$. By Theorem \ref{N^+},
\[N^+=\frac{1}{n} \left[\begin{array}{r|c}
\bm{1} & X \\
\hline
\bm{0} & Y
\end{array}\right],\]
where $X=2(CC^T+I_{n-1})^{-1}\left[J_{n-1}-nI_{n-1}\right]$ and $Y=-C^TX$.

\begin{align*}
    L^+ & = (N^+)^TN^+ \\
    & = \frac{1}{n^2}\left[\begin{array}{c|c}
    \bm{1}^T & \bm{0}^T \\
    \hline
    X^T & Y^T
    \end{array}\right]\left[\begin{array}{c|c}
    \bm{1} & X \\
    \hline
    \bm{0} & Y
    \end{array}\right] \\
    & = \frac{1}{n^2}\left[\begin{array}{c|c}
    \bm{1}^T & \bm{0}^T \\
    \hline
    X & Y^T
    \end{array}\right]\left[\begin{array}{c|c}
    \bm{1} & X \\
    \hline
    \bm{0} & Y
    \end{array}\right]\;\;(\text{since $X$ is symmetric}) \\
    & = \frac{1}{n^2}\left[\begin{array}{c|c}
    \bm{1}^T\bm{1} & \bm{1}^TX\\
    \hline
    X\bm{1} & XX+Y^TY
    \end{array}\right] \\
    & = \frac{1}{n^2}\left[\begin{array}{c|c}
    n-1 & -\bm{1}^T \\
    \hline
    -\bm{1} & XX+Y^TY
    \end{array}\right]\;\;(\text{since $\bm{1}^TX=-\bm{1}^T$ and $X\bm{1}=-\bm{1}$}) \tag{4}\label{eq4}
\end{align*}

Now we simplify $XX+Y^TY$ as follows.

\vspace{-12pt}
\begin{align*}
    XX+Y^TY & = XX-XC(-C^TX) \\
    & = XI_{n-1}X+XCC^TX \\
    & = X(CC^T+I_{n-1})X \\
    & = X(CC^T+I_{n-1})(CC^T+I_{n-1})^{-1}(J_{n-1}-nI_{n-1}) \\
    & = X(J_{n-1}-nI_{n-1}) \\
    & = -J_{n-1}-nX
\end{align*}

Plugging $XX+Y^TY=-J_{n-1}-nX$ in (\ref{eq4}), we get

\[L^+ = \frac{1}{n^2}\left[\begin{array}{c|c}
    n-1 & -\bm{1}^T \\
    \hline
    -\bm{1} & -J_{n-1}-nX
    \end{array}\right],\]
where $X$ is given by Corollary \ref{X=circ2}.
\end{proof}

\begin{example}
Consider $W_6$ with vertex and edge labeling given in Figure \ref{fig:OW6} and its  Laplacian matrix $L$. The Moore-Penrose inverse $L^+$ of $L$ is as follows.

\[L=\left[\begin{array}{r|rrrrr}
5 & -1 & -1 & -1 & -1 & -1 \\
\hline
-1 & 3 & -1 & 0 & 0 & -1 \\
-1 & -1 & 3 & -1 & 0 & 0 \\
-1 & 0 & -1 & 3 & -1 & 0 \\
-1 & 0 & 0 & -1 & 3 & -1 \\
-1 & -1 & 0 & 0 & -1 & 3
\end{array}\right],\;
L^+=\frac{1}{396}\left[\begin{array}{r|rrrrr}
55 & -11 & -11 & -11 & -11 & -11 \\
\hline
-11 & 103 & -5 & -41 & -41 & -5 \\
-11 & -5 & 103 & -5 & -41 & -41 \\
-11 & -41 & -5 & 103 & -5 & -41 \\
-11 & -41 & -41 & -5 & 103 & -5 \\
-11 & -5 & -41 & -41 & -5 & 103
\end{array}\right].\]
\end{example}

\end{document}